\documentclass{article}

\usepackage{amsmath,amsfonts,amsthm,amscd,amssymb,graphicx}

\usepackage{subfigure}

\numberwithin{equation}{section}
\usepackage{color}

\newtheorem{theorem}{Theorem}[section]

\newtheorem{proposition}[theorem]{Proposition}
\newtheorem{prop}[theorem]{Proposition}

\def\eps{\varepsilon }

\def\beq{\begin{equation}}
\def\eeq{\end{equation}}
\def\bb1{{1\!\!1}}
\def\rit{{\Bbb R}}

\def\eps{\varepsilon}

\begin{document}

%%% title: ±êÌâ
%%%   \title{title}{title for citation}
\centerline{\bf \Large Instabilities  of shear layers}

\bigskip

\bigskip

\centerline{
Dongfen Bian \footnote{Beijing Institute of Technology, School of Mathematics and Statistics, Beijing {\rm 100081}, China, biandongfen@bit.edu.cn},
Emmanuel Grenier \footnote{UMPA, CNRS  UMR $5669$,Ecole Normale Sup\'erieure de Lyon, Lyon, France, Emmanuel.Grenier@ens-lyon.fr}
}

\bigskip

\bigskip

\centerline{ \it December $2023$}

%%%%%

\subsubsection*{Abstract}

%%%%%

This article gathers notes of two lectures given at Grenoble's University in June $2023$, and is an introduction to recent 
works on shear layers, in collaboration with D. Bian, Y. Guo,  T. Nguyen and B. Pausader.

\medskip

\noindent Updated versions will be posted on Arxiv.

\tableofcontents

\newpage

%%%%%%%%%%%%%%%%%%%%%%%

\section{Introduction}

%%%%%%%%%%%%%%%%%%%%%%%

In this paper, we are interested in the inviscid limit of the incompressible Navier Stokes equations with Dirichlet boundary condition in the half
plane. The classical incompressible Navier Stokes equations in the half space $\Omega = \{ y > 0 \}$ read
\beq \label{NS1}
\partial_t u^\nu + (u^\nu \cdot \nabla) u^\nu + \nabla p^\nu - \nu \Delta u^\nu = 0,
\eeq
\beq \label{NS2}
\nabla \cdot u^\nu = 0,
\eeq
\beq \label{NS3}
u^\nu = 0 \quad \hbox{when} \quad y = 0.
\eeq
As $\nu$ goes to $0$, formally, Navier Stokes equations degenerate into the incompressible Euler equations
\beq \label{Eu1}
\partial_t u + (u \cdot \nabla) u + \nabla p  = 0,
\eeq
\beq \label{Eu2}
\nabla \cdot u = 0,
\eeq
\beq \label{Eu3}
u \cdot n = 0 \quad \hbox{when} \quad y = 0,
\eeq
where $n$ is a vector normal to the boundary.

Despite many efforts, the question of the convergence of solutions of Navier Stokes equations to solutions of Euler equations
is widely open. A classical approach is to introduce the so called Prandtl's boundary layer, of size $O(\sqrt{\nu})$. This leads
to look for $u^\nu$ under the form 
\beq \label{Prandtl}
u^\nu(t,x,y) = u^{Euler}(t,x,y) + u^{Prandtl}(t,x,\nu^{-1/2} y) + O(\sqrt{\nu})_{L^\infty},
\eeq
where $u^{Euler}$ solves Euler equations and $u^{Prandtl}$ is a corrector, located near the boundary, solution of the so called
Prandtl's equations (see \cite{Reid} for instance). Such an expansion is however far from being fully justified.
Roughly speaking, the attempts to answer this question follow four different directions

\begin{itemize}

\item Proof of (\ref{Prandtl}) in small time for initial data with analytic type regularity. This direction has been pioneered by 
R.E. Caflish and M. Sammartino \cite{Caf1,Caf2} and has later been extended to data with Gevrey regularity, or to initial data
with vorticity supported away from the boundary. These results in particular show that the formal analysis of Prandtl is valid in the particular case
of analytic data \cite{DGV2,Maekawa}.
In that direction, note also the construction of stationary solutions for Blasius' boundary layer \cite{Blasius1},\cite{Blasius2}. 
We in particular refer the readers to the works of N. Masmoudi, D. G\'erard-Varet, Y. Maekawa, Y. Guo, V. Vicol.

\item Proof that (\ref{Prandtl}) does not hold because Prandtl's equation has singularities: blow up \cite{EE}, 
boundary separation \cite{DM}.

\item Proof that (\ref{Prandtl}) is false because of underlying linear instabilities in the case of solutions with Sobolev regularity. 
These instabilities lead to the successive onset of various sublayers, of size $\nu^{1/2}$ as expected, but also of sizes
$\nu^{3/4}$ or $\nu^{13/16}$ as will be detailed in this article. 
 This includes the proof of the linear instability of shear flows \cite{GGN3}, 
 the instability of Prandtl layers in $L^\infty$ for a particular class of initial data \cite{GN1},
and the instability of Blasius boundary layer under Sobolev perturbations \cite{Blasius3}.

\item Kato's result, which  is a singular and isolated result \cite{Kato}:
$u^\nu$ converges to $u$ in $L^2([0,T] \times \Omega)$ if and only if
$$
\nu \int_0^T \int_{d(x,\partial \Omega) \le  \nu} \| \nabla u^\nu \|^2 \, dx \, dt \to 0
$$
as $\nu$ goes to $0$.
Namely, solutions of Navier Stokes equations converge to a solution of Euler equation if the energy dissipation in a layer of size
$O(\nu)$ goes to $0$. 
A striking feature of Kato's result is that the size which appears is $\nu$ and not $\sqrt{\nu}$, the classical size of Prandtl boundary
layer. Therefore the question of the convergence of solutions of Navier Stokes equations is linked to the capability of these solutions
to create boundary layers of size $\nu$, starting from a classical Prandtl layer of size $\sqrt{\nu}$.

\end{itemize}

In a paradoxical way, the first two approaches manage to justify classical results from Physics, namely the importance
of Prandtl boundary layer \cite{Caf1,Caf2}, or a classical scenario of boundary separation \cite{DM}, but at the expense of
the analyticity of the solution, or its time independence. 

%Blasius' boundary layer are in fact
%unstable with respect to small Sobolev perturbations, both linearly and nonlinearly. Moreover, to assume that
%the initial data is analytic is physically too strong, and in any case the results only hold in small time.

\medskip

The third and fourth approaches are also paradoxical, since, whereas the requirements on the smoothness of the initial data
are natural ($H^s$ or $L^2$ regularity), the boundary layers which appear in these studies are unphysical, like $\nu^{13/16}$ or $\nu$. 
They do not appear in the classical physics literature and seem mere mathematical artefacts.

\medskip

The aim of these lectures is to discuss the third approach and to discuss this paradox.
The plan is the following. In section $2$ we introduce the Rayleigh and Orr Sommerfeld equations, which are studied in more details in 
section $3$ and $4$, where we obtain optimal bounds on the solutions to linearised Euler and Navier Stokes equations.
Sections $5$ and $6$ are devoted to the study of two different kinds of instabilities of shear flows. In section $7$ we discuss
multilayers instabilities, and then give perspectives in section $8$.

\medskip

Before jumping into the details we state the two main results that are discussed in these lectures.
We will consider particular solutions of the Navier Stokes equations called "shear layers", namely solutions $V_0^\nu(t,y)$ of the form
\beq \label{heat0}
V_0^\nu(t,y) = \left(
\begin{array}{cc}
U_s(t,\nu^{-1/2} y) \cr
0 
\end{array} \right),
\eeq
where $U_s$ is a given smooth function.
For such flows, Navier Stokes equations reduce to the heat equation
\beq \label{heat1}
\partial_t U_s - \partial_{yy} U_s = 0,
\eeq
together with the boundary condition 
\beq \label{heat2}
U_s(t,0) = 0 .
\eeq
Note in particular that $V_0^\nu$ is analytic for any time, provided $U_s(0,\cdot)$ is analytic in $y$. We define
$$
U_+ = \lim_{y \to \infty} U_s(0,y),
$$
and will assume that $U_+ \ne 0$.

It turns out that non trivial shear layers are always linearly unstable provided the viscosity is small enough \cite{GGN3}.
The nature of the instability depends on whether the profile is stable for linearised Euler equations or not (a notion detailed below).
This leads to the following two different results:

\begin{theorem} \label{theomain}
Let $U_s(0,\cdot)$ be a smooth profile which is spectrally unstable for Euler equations. Then
for any $N$ and $s$ arbitrarily large, there exists a sequence of solutions $V^\nu$ of Navier Stokes equations
with forcing term $F^\nu$ and a sequence of times $T^\nu$ such that
\beq 
\| V^\nu(0,\cdot,\cdot) - V_0^\nu(0,\cdot) \|_{H^s} \le \nu^N,
\eeq
\beq 
\| F^\nu \|_{L^\infty([0, T^\nu], H^s)} \le \nu^N,
\eeq
\beq
\| V^\nu(T^\nu,\cdot,\cdot) - V_0^\nu(T^\nu,\cdot) \|_{L^\infty} \ge \sigma  > 0,
\eeq
\beq 
T^\nu \sim C_0 \log \nu^{-1} ,
\eeq
and such that, as $t \to T^\nu$,  $V^\nu(1\cdot,\cdot)$ exhibits the scales $\nu^{1/2}$, $\nu^{3/4}$ and $\nu^{13/16}$ in the $y$ variable.
\end{theorem}

If we start with a profile $U_s$ which is spectrally stable for Euler equations, for instance a monotonic and concave profile, 
the instability process is different
and we have
\begin{theorem} \label{theofirst2}
Let $U_s(0,\cdot)$ be a smooth profile which is spectrally stable for Euler equations. Then
for any $N$ and $s$ arbitrarily large, there exists a sequence of solutions $V^\nu$ of Navier Stokes equations
with forcing term $F^\nu$,
and a sequence of times $T^\nu$ such that
\beq 
\| V^\nu(0,\cdot,\cdot) - V_0^\nu(0,\cdot) \|_{H^s} \le \nu^N,
\eeq
\beq 
\| F^\nu \|_{L^\infty([0, T^\nu], H^s)} \le \nu^N,
\eeq
\beq
\| V^\nu(T^\nu,\cdot,\cdot) - V_0^\nu(T^\nu,\cdot) \|_{L^\infty} \ge \sigma \nu^{1/4} > 0,
\eeq
\beq 
T^\nu \sim C_0 \log \nu^{-1} ,
\eeq
and such that, as $t \to T^\nu$, $u^\nu$ exhibits the scales $\nu^{1/2}$ and $\nu^{5/8}$ in the $y$ variable. 
\end{theorem}

These Theorems are discussed in sections $5$, $6$ and $7$. For the proofs, we refer to \cite{GGN3} and \cite{Bian4}.

%%%%%%%%%%

\section{Fourier Laplace transform}

%%%%%%%%%%

%%%%%

\subsection{Warm up}

%%%%%

Let us first recall Laplace's method.
Let $A$ be a matrix, and let $\phi(t)$ be the solution of
$$
{d \phi \over dt} (t) = A \phi(t),
$$
with initial data $\phi(0) = \phi_0$. 
Then we have the following representation formula
\beq \label{representation}
\phi(t) = e^{t A} \phi_0 = {1 \over 2 i \pi}  \int_\Gamma  e^{\lambda t} (A - \lambda )^{-1} \phi_0 \, d\lambda ,
\eeq
where $\Gamma$ is an integration contour "on the right" of the spectrum of $A$.
Using analyticity we can then play with this contour, and shift it to the left till it reaches the point spectrum
of $A$. We can further shift it to the left, going through the point spectrum of $A$, provided we add the residues at these singularities.
According to Cauchy's formula, this leads to
$$
\phi(t) = \sum_{\lambda \in Sp(A)} P_\lambda(\phi_0) e^{\lambda t} 
+  {1 \over 2 i \pi} \int_\Gamma  ( A - \lambda  )^{-1} \phi_0 \, d\lambda
$$
where $Sp(A)$ is the spectrum of $A$, 
$P_\lambda$ is the projection on the  eigenspace corresponding to the eigenvalue $\lambda$ (assuming that the eigenvalue is simple
to simplify the presentation),
and where the contour $\Gamma$ is now on the left of the spectrum. 
Note that the corresponding integral is  now exponentially decaying, with an arbitrary fast
speed (and in fact vanishes in the case of matrices).

In these notes we will follow exactly the same strategy with $A$ replaced by linearised Euler or Navier Stokes operators.

%%%%%%

\subsection{Orr Sommerfeld and Rayleigh equations}

%%%%%%

In this section we introduce the classical Rayleigh and Orr Sommerfeld equations.
Let $L_{NS}$ be the linearised Navier Stokes operator near a shear layer profile $U_s(y)$, namely
\begin{equation} \label{linearNS}
L_{NS} v = \Bigl( (U_s(y),0)  \cdot \nabla \Bigr) v + \Bigl( v \cdot \nabla \Bigr) (U_s(y),0) - \nu \Delta v + \nabla q,
\end{equation}
with $\nabla  \cdot v = 0$ and Dirichlet boundary condition.
Let $v(t)$ be a solution of
$$
\partial_t v(t) + L_{NS}v(t) = 0 
$$
with initial data $v(0) = v_0$.
Then, as for (\ref{representation}), we have the representation formula
\beq \label{LaplaceNS}
v(t) = {1 \over 2 i \pi} \int_\Gamma e^{\lambda t} \, (L_{NS} - \lambda )^{-1} v_0 \, d\lambda,
\eeq
where $\Gamma$ is any contour on the right of the spectrum of $L_{NS}$.
We thus need to study the resolvent of $L_{NS}$, namely to study the equation
\begin{equation} \label{resolvant}
(L_{NS} - \lambda) u = v_0.
\end{equation}
We take advantage of the incompressibility condition and introduce the stream function $\psi(t,x,y)$ of $u(t,x,y)$ in such a way that
$$
u = \nabla^\perp \psi .
$$
We further take the Fourier transform $\psi_\alpha(t,y)$ of $\psi(t,x,y)$ in the horizontal variable $x$, with dual Fourier variable $\alpha \in \rit$.
Following the tradition \cite{Reid}, we  define $c$ by
$$
\lambda = - i \alpha c.
$$
This leads to write the horizontal Fourier transform $u_\alpha(t,y)$ of $u(t,x,y)$ under the form
$$
u_\alpha = \nabla^\perp \Bigl( e^{i \alpha x } \psi_\alpha(y) \Bigr)=e^{i \alpha x }(\partial_y \psi_\alpha, -i\alpha \psi_\alpha).
$$
Note that the horizontal Fourier transform $\omega_\alpha$ of the vorticity $\omega$ is simply given by
\beq \label{omegapsi}
\omega_\alpha = -  (\partial_y^2 - \alpha^2) \psi_\alpha .
\eeq
Taking the curl of (\ref{resolvant}) we then get
\begin{equation} \label{Orrmod}
OS_{\alpha,c,\nu}(\psi) :=  (U_s - c)  (\partial_y^2 - \alpha^2) \psi 
 - U_s''  \psi  
- { \nu \over i \alpha}   (\partial_y^2 - \alpha^2)^2 \psi =- {\omega_{0,\alpha} \over i \alpha},
\end{equation}
where 
$$
\omega_{0,\alpha} = (\nabla \times v_0)_\alpha
$$ 
denotes the horizontal Fourier transform of the curl of the initial data $v_0$.
The Dirichlet boundary condition
gives 
\begin{equation} \label{condOrr}
\psi_\alpha(0) = \partial_y \psi_\alpha(0) = 0.
\end{equation}
Moreover, $\psi_\alpha(y) \to 0$ as $y \to + \infty$.
We also introduce
\begin{equation} \label{epsilon}
\varepsilon = {\nu \over i \alpha} .
\end{equation}
As $\nu$ goes to $0$, the Orr Sommerfeld operator $OS_{\alpha,c,\nu}$ degenerates into the classical Rayleigh operator
\begin{equation} \label{Rayleigh}
Ray_{\alpha,c}(\psi) = (U_s - c)  (\partial_y^2 - \alpha^2) \psi 
 - U_s''  \psi
 \end{equation}
 which is a second order operator, together with the boundary condition 
 \begin{equation} \label{condRay}
 \psi(0) = 0.
 \end{equation}
 The next two sections are devoted to a detailed study of Rayleigh (section $3$) and Orr Sommerfeld (section $4$), which allow to obtain
 optimal bounds on the solutions to linearised Euler and Navier Stokes equations thanks to (\ref{LaplaceNS}).

 %%%%%%%%%%%%%%%%%

 \section{Rayleigh equation}

 %%%%%%%%%%%%%%%%%

%%%%

\subsection{Link with linearised Euler equations}

%%%%

Note that $Ray_{\alpha,c}^{-1}$ is the resolvent of the linearised Euler operator, in stream function formulation and after an horizontal Fourier transform.
Namely, let $L_E$ denote the linearised Euler operator
$$
L_E u = \Bigl( (U_s(y),0)  \cdot \nabla \Bigr) u +  \Bigl( u \cdot \nabla \Bigr) (U_s(y),0) + \nabla p,
$$
and let $u(t)$ be the solution of 
$$
\partial_t u = L_E u
$$
with initial data $u^0$. Let $\psi_\alpha^0$ be the Fourier transform of the stream function of  the initial data $u^0$,
and, with a slight abuse of notation, let us denote by $e^{L_E t} \psi_\alpha^0$ the stream function of $u(t)$.
Then, using $\lambda = - i \alpha c$, we have
\beq \label{resolEuler}
e^{L_E t} \psi_\alpha^0 = - {1 \over 2 i  \pi} \int_\Gamma e^{- i \alpha c t} Ray_{\alpha,c}^{-1} \, \omega_{0,\alpha}  \, \, dc
\eeq
where $\Gamma$ is a contour "above" the spectrum of $Ray_{\alpha,c}$
(as $\lambda = - i \alpha c$,  the contour is now "above" the spectrum, instead of "on the right").

We note that (\ref{resolEuler}) gives an explicit expression of the solution of linearised Euler equations, in terms of the inverse
of Rayleigh operator. As we will see, a detailed analysis of the Rayleigh operator allows to give optimal bounds on the behavior
of solutions of linearised Euler equations, and in particular to investigate the so called "inviscid damping". 

The first step is to study the spectrum of linearised Euler equations.
The spectrum $\sigma_{Euler}$ of $L_E$ is composed of a point spectrum $\sigma_P$ and of a continuous spectrum $\sigma_C$.
We note that $Ray_{\alpha,c}$ degenerates when $U_s(y) = c$, hence the continuous spectrum $\sigma_C$ is simply the range of $U_s$.
Now, as $U_s$ has an exponential behavior, for $c$ out of the range of $U_s$,
there exists only one solution $\psi_{\alpha,c}(y)$ of (\ref{Rayleigh})  such that
$\psi_{\alpha,c}(y) \sim e^{- | \alpha | y}$ as $y \to + \infty$. We thus have
$$
\sigma_{Euler} = \sigma_P \cup \sigma_C, \qquad
\sigma_P = \Bigl\{ c \quad | \quad \psi_{\alpha,c}(0) = 0 \Bigr\}, \qquad 
\sigma_C = Range(U_s).
$$
As $\psi_{\alpha,c}(0)$ is an holomorphic function of $c$ outside the range of $U_s$, the eigenvalues are isolated. 
It is even possible to prove that their number is finite.

Let us first recall Rayleigh's criterium:

\begin{prop}
If $U_s$ has no inflection point, then $\sigma_P \cap \{ \Im c > 0 \} = \emptyset$.
\end{prop}

\begin{proof}
Let $\psi$ and $c$ be an eigenvector and eigenvalue of Rayleigh operator. Then
$$
(\partial_y^2 - \alpha^2) \psi = {U_s'' \over U_s - c} \psi .
$$
We multiply by $\bar \psi$, which gives
$$
- \int | \partial_y \psi |^2 + \alpha^2 | \psi |^2 \, dy = \int {U_s'' \over U_s - c} | \psi |^2 \, dy .
$$
We then take the imaginary part, which gives
$$
\Im c \int  U_s''   { | \psi |^2 \over | U_s - c |^2 } \, dy  = 0 .
$$
Thus $U_s''$ must change sign, hence $U_s$ must have an inflection point.
\end{proof}

In particular, concave profiles are spectrally stable (i.e., $\sigma_P \cap \{ \Im c > 0 \} = \emptyset$). Rayleigh's criterium has been
refined by Fjortoft (see \cite{Reid}). Up to now, there is no necessary and sufficient condition for a profile $U_s$ to be spectrally stable.
In particular some profiles with inflection points are stable.

%%%%%%

\subsection{Construction of two particular solutions of Rayleigh equation}

%%%%%%

For the simplicity of the presentation we restrict ourselves to profiles $U_s$ which are concave, with $U_s'(y) > 0$.
If $c$ is away from the range of $U_s$, then $U_s(y) - c$ does not vanish. Rayleigh equation is then a regular ordinary differential equation.
We thus just focus on the case where $c$ is close to the range of $U_s$, where Rayleigh equation degenerates.
For $c$ close enough to $Range(U_s)$, there exists some complex number $y_c$ such that 
$$
U_s(y_c) = c.
$$
As $U_s(y)$ is monotonic in $y$, for a given real $c$, there exists at most one such $y_c$, which simplifies the analysis.
Such a complex number $y_c$ is called a "critical layer". We note that Rayleigh operator $Ray_{\alpha,c}$ degenerates at this critical layer.
A first solution of (\ref{Rayleigh}) is obtained by looking for an holomorphic solution which vanishes at $y_c$.
Looking for this first solution $\psi_{\alpha,c}^r(y)$ under the form
$$
\psi_{\alpha,c}^r(y) = \sum_{n \ge 1} \beta_n(\alpha,c) (y - y_c)^n,
$$
it is easy to get a recurrence relation on the coefficients $\beta_n(\alpha,c)$, depending on the Taylor expansion of $U_s$ near $y_c$, and to prove that
the corresponding series has a positive radius of convergence. Then $\psi_{\alpha,c}^r$ may be extended to a "regular" solution
on the whole real line.

An independent solution is obtained through the method of the variation of the coefficients, and is explicitly given by
$$
\psi_{\alpha,c}^s(y) = \psi_{\alpha,c}^r(y) \int_{y^\star}^y {dz \over \psi_{\alpha,c}^r(z)^2} ,
$$
where $y^\star$ can be chosen arbitrarily.
Expanding $\psi_{\alpha,c}^r$ near $y_c$ in this expression, we observe that, near $y_c$, this second solution behaves like
$$
\psi_{\alpha,c}^s(y) = P_{\alpha,c}^s(y) + Q_{\alpha,c}^s(y) \, (y - y_c) \, \log(y - y_c)
$$
where $P_{\alpha,c}^s$ and $Q_{\alpha,c}^s$ are holomorphic functions near $y_c$.
It is important to note that while $\psi_{\alpha,c}^r$ is smooth on the real line, $\psi_{\alpha,c}^s$ is non holomorphic at $y_c$, with a
"$(y - y_c) \log(y - y_c)$" singularity. This is a "small singularity" for the stream function, however the corresponding vorticity 
behaves like
$$
\omega_{\alpha,c}^s(y) \sim {Q_{\alpha,c}^s(y_c) \over y - y_c} 
$$
and is thus singular at $y_c$. Of course all the solutions of $Ray_{\alpha,c} \psi = 0$ are a linear combination of $\psi_{\alpha,c}^r$
and $\psi_{\alpha,c}^s$.
Physically, at $y_c$, there is a "resonance" between the speed of the flow $U_s(y_c)$ and the wave, leading to this singularity on the vorticity.

Now, using linear combinations of $\psi_{\alpha,c}^r$ and $\psi_{\alpha,c}^s$, it is possible to construct $\psi_{\pm,\alpha}(y,c)$
such that
$$
\psi_{\pm, \alpha}(y,c) \sim e^{\pm | \alpha | y} 
$$
as $y \to + \infty$, with unit Wronskian. We note that both $\psi_{\pm,\alpha}(y,c)$ may be singular at $y_c$ and are of the form 
\beq \label{exppsi}
\psi_{\pm,\alpha}(y,c) = P_{\pm,\alpha}(y - y_c,c) + (y - y_c) \log(y - y_c) Q_{\pm,\alpha}(y - y_c,c) 
\eeq
near $y_c$,
where $P_{\pm,\alpha}$ and $Q_{\pm,\alpha}$ are holomorphic. Of course $\psi_{-,\alpha}(y,c)$ goes to $0$ as $y \to + \infty$,
while $\psi_{+,\alpha}(y,c)$ diverges at infinity. We note that $\psi_{\pm,\alpha}(y,c)$ have "logarithmic branches" at
$y = y_c$.

%%%%%

\subsection{Green function for Rayleigh equation}

%%%%%

Using $\psi_{\pm,\alpha}(y,c)$, it is possible to explicitly construct the Green function of $Ray_{\alpha,c}$, defined by
\beq \label{defiG0c}
Ray_{\alpha,c} \Bigl( G_{\alpha,c}(x,y) \Bigr) = \delta_x
\eeq
and by $G_{\alpha,c}(x,0) = 0$ and $G_{\alpha,c}(x,y) \to 0$ as $y \to + \infty$. 
For this we first solve (\ref{defiG0c}) without the boundary conditions, just choosing a particular solution.
 This leads to the introduction of
 $$
 G_{\alpha,c}^{int}(x,y) = {1 \over U_s(x) - c} 
 \Bigl\{ \begin{array}{c}
 \psi_{-,\alpha}(y,c) \psi_{+,\alpha}(x,c) \qquad \hbox{if} \qquad y > x \cr
 \psi_{-,\alpha}(x,c) \psi_{+,\alpha}(y,c) \qquad \hbox{if} \qquad y < x , \cr
\end{array} 
 $$
 keeping in mind that the Wronskian between $\psi_{+,\alpha}$ and $\psi_{-,\alpha}$ is one.
 We then correct $G_{\alpha,c}^{int}$ by $G_{\alpha,c}^b$ such that 
 $$
 G_{\alpha,c} = G_{\alpha,c}^{int} + G_{\alpha,c}^b
 $$ 
 satisfies both Rayleigh's equation and the two boundary conditions. This leads to 
$$
 G_{\alpha,c}^b(x,y) = - \psi_{+,\alpha}(0,c) {\psi_{-,\alpha}(y,c) \over \psi_{-,\alpha}(0,c) }
 {\psi_{-,\alpha}(x,c) \over U_s(x) - c} 
$$
We note that $G_{\alpha,c}$ is singular at $c = U_s(y)$ and $c = U_s(x)$, with poles and "logarithmic branches"
at this points.

The solution of Rayleigh's equation
\beq \label{Ray0cf}
 (U_s - c)  (\partial_y^2 - \alpha^2 ) \psi - U_s''  \psi  = f,
 \eeq
 together with its two boundary conditions, is then explicitly given by
 \beq \label{constructionpsi}
 \psi(y) = \int_0^{+\infty} G_{\alpha,c}(x,y) f(x) \, dx .
 \eeq

%%%

\subsection{Inviscid damping}

%%%

We now combine (\ref{resolEuler}) and (\ref{constructionpsi}) to study the asymptotic behavior of solutions to Euler equations.
Bounds on solutions of linearised Euler equations have been pioneered in physics by F. Bouchet and H. Morita in \cite{Bouchet},
and later rigorously proved in \cite{Bedrossian-Zelati-Vicol-2019,Ionescu-Jia-CMP-2020,Ionescu-Iyer-Jia,Wei-Zhang-Zhao-CPAM-2018,
Wei-Zhang-Zhao-APDE-2019,Wei}. The main feature is that solutions of linearised Euler equations decay like $t^{-1}$ if
there is no unstable eigenvalue, a phenomena known as "inviscid damping".

In this section, following the lines of \cite{Bian2}, we prove the inviscid damping in a simplified case by combining  
(\ref{resolEuler}) and (\ref{constructionpsi}).
To simplify the presentation, we assume that $U_s(y)$ is holomorphic in $y$, satisfies $U_s(0) = 0$, and is monotonically increasing in $y$,
concave, with $U_s'(y) > 0$.

\begin{theorem} \label{growthRayleigh}
Assume that $U_s(y)$ is holomorphic in $y$, satisfies $U_s(0) = 0$, and is monotonically increasing and concave in $y$.
Let $\langle t \rangle = 1 + t$. Assume that the initial data $\omega_0(x)$ is a smooth function, then
\beq \label{result2}
| \alpha | \| \psi_\alpha(t,\cdot) \|_{L^\infty}  +  \| \partial_y \psi_\alpha(t,\cdot) \|_{L^\infty}   \lesssim \langle t \rangle^{-1}.
 \eeq
 Moreover there exists $\omega_\infty(t,y)$ such that
 \beq \label{result4}
 \omega_\alpha(t,y) = \omega_\infty(t,y) e^{- i \alpha U_s(y) t} + O(\langle t \rangle^{-1}).
 \eeq
\end{theorem}

Note that, as $U_s$ is monotonically increasing $\sigma_P \cap \{ \Im c > 0 \} = \emptyset$. Hence there is no unstable mode.
The first step is to move $\Gamma$ close to $\sigma_C$ in (\ref{resolEuler}) and to choose for instance
$$
\Gamma = \{ (1 + i) \rit_- + m + i \eps\}  \cup \{ [m, M] + i \eps \}  \cup \{ (1 - i) \rit_+ + M + i \eps \} 
$$
for some small $\eps$.
However, with this choice of $\Gamma$ we only get an exponential bound, of the form $e^{\eps t}$.
To get the $\langle t \rangle^{-1}$ decay, we need to go "through" $\sigma_C$, 
and hence to study how Rayleigh operator behaves close to $Range(U_s)$.
Let 
$$
\psi_{\alpha,c} = - (i\alpha)^{-1} \,  Ray_{\alpha,c}^{-1} \, \, \omega_{0,\alpha}.
$$
Then $\psi_{\alpha,c}$ is explicitly given using  formula (\ref{constructionpsi}), and we have
$$
e^{L_E t} \psi_\alpha^0 = {\alpha \over 2 \pi} \int_\Gamma e^{- i \alpha c t} \psi_{\alpha,c}(y) \, dc .
$$
We further observe that the corresponding vorticity 
$$
\omega_{\alpha,c} = - (\partial_y^2 - \alpha^2) \psi_{\alpha,c}
$$ 
is directly given by Rayleigh's equation
\beq \label{defiomega}
\omega_{\alpha,c} = {i  \alpha U_s''  \psi_{\alpha,c} + \omega_{\alpha,0} \over \alpha (U_s - c) }.
\eeq
Thus the horizontal Fourier transform $\omega_\alpha(t)$ of the solution $u(t)$ is given by
\beq \label{omea}
\omega_\alpha(t,y) = {\alpha \over 2  \pi } \int_{\Gamma} e^{- i \alpha c t}  {\zeta_\alpha(y,c) \over  U_s(y) - c } \; dc 
\eeq
where
\beq \label{zeta}
\zeta_{\alpha,c} = i \alpha U_s'' \psi_{\alpha,c}  +  \omega_{\alpha,0}.
\eeq
We will first discuss the following result

\begin{proposition}
Under the assumptions of Theorem \ref{growthRayleigh}, as $t \to + \infty$,
\beq \label{vorticityasymp}
\omega_\alpha(t,y) = % \omega_{unstable}(t,y) +  
\zeta_\alpha(y) e^{- i \alpha U_s(y) t} + O(t^{-1}).
\eeq
%where $\omega_{unstable}$ is the projection on possible unstable modes and  is exponentially growing.
\end{proposition}

\begin{proof}
% The first step is to move the contour through the point spectrum $\sigma_P(U_s) \cap \{ \Re c > 0\}$,
% which gives
% \beq \label{remain2}
%\omega_\alpha(t,y) = \omega_{unstable}(t,y) +  \omega_{remain}(t,y)
%\eeq
%where
%\beq \label{remain}
%\omega_{remain}(t,y) =  {\alpha \over 2   \pi } \int_\Gamma e^{- i \alpha c t} 
% {\zeta_\alpha(y,c) \over  U_s(y) - c}\; dc ,
%\eeq
% where $\omega_{unstable}$ is the projection on the unstable eigenvalues 
%and where now $\Gamma$ is "below"  of the unstable eigenvalues but "above"  the range of $U_s$.
We split $\zeta_\alpha$ according to (\ref{zeta}).
Assuming that $\omega_{\alpha,0}$ is holomorphic, we can move $\Gamma$ downwards to $\Im c < 0$,
and, thanks to Cauchy's residue theorem, the contribution of the pole at $c = U_s(y)$ gives
$$
 {\alpha \over 2   \pi } \int_\Gamma e^{- i \alpha c t} 
 {\omega_{\alpha,0}(y) \over  U_s(y) - c}\; dc 
 = e^{- i \alpha U_s(y) t} \omega_{\alpha,0}(y) + O(e^{- \eps t}) 
$$
for some positive $\eps$. The other term is
\beq \label{defiI}
I :=  {\alpha \over 2   \pi } \int_\Gamma e^{- i \alpha c t}  U_s''(y) {\psi_{\alpha}(y,c) \over  U_s(y) - c}\; dc .
\eeq
We then insert
$$
\psi_{\alpha}(y,c) = \int G_\alpha(x,y) \psi_\alpha^0(x) \, dx
$$
in (\ref{defiI}).
This leads to various terms, and amongst them
$$
\int_{x < y} \int_\Gamma U_s''(y) {\psi_{+,\alpha}(x,c) \over U_s(x) - c}  {\psi_{-, \alpha}(y,c) \over U_s(y) - c}  \, dc \, dx.
$$
Let us fix $x$ and $y$. We have to study
$$
I := \int_\Gamma  {\psi_{+,\alpha}(x,c) \over U_s(x) - c}  {\psi_{-, \alpha}(y,c) \over U_s(y) - c}  \, dc 
$$
We note that the integrand is singular when $c = U_s(y)$ or $c = U_s(x)$, with two types of singularities: poles and "logarithmic branches".
When we shift $\Gamma$ downwards, the poles give contributions through Cauchy's residue theorem.
The "logarithmic branch" lead to consider terms of the form
$$
J := \int_\Gamma R(c) \log( y - y_c) 
$$
where $R(c)$ is holomorphic near $y_c$.
We choose the determination of the logarithmic so that $c \to \log(y - y_c)$ is well defined except on a vertical half line.
To bound $J$ we shift $\Gamma$ downwards, making the "turn" of the singularity through two nearby vertical lines, which gives
a $t^{-1}$ contribution. The polynomial decay thus arises from the logarithmic branch created by Rayleigh's singularity.
\end{proof}

We now go back to Theorem \ref{growthRayleigh}.
To go from the vorticity to the stream function we need to invert the Laplace operator $\partial_y^2 - \alpha^2$.
Let $H_\alpha(x,y)$ be the Green function of $(\partial_y^2 - \alpha^2)$ with Dirichlet boundary condition.
We have
$$
H_{\alpha}(x,y) = {e^{- | \alpha | | x- y |} \over 2 | \alpha |} - {e^{- | \alpha | | x + y |} \over 2 | \alpha |} .
$$
Then
\beq \label{inverseomega}
\psi_\alpha(x) = \int_0^{+\infty} H_\alpha(x,y) \omega_{\alpha}(t,y) \, dx 
= \int_0^{+\infty} H_\alpha(x,y) \zeta_{\alpha}(y) e^{- i \alpha U_s(y) t} \, dx + O(t^{-1})
\eeq
and the end of the computations are straightforward.

%%%%%%%%%%%%%%%

\section{Orr Sommerfeld equation}

%%%%%%%%%%%%%%%

Let us now study the Orr Sommerfeld operator.  We refer to \cite{GGN3} for detailed proofs. We only qualitatively describe the
construction of solutions to Orr Sommerfeld equation in this section, without details and give hints on the long time behavior of solutions
to linearised Navier Stokes equations (see \cite{Bian4} for more details).

%%%%%%%

\subsection{Construction of four independent solutions}

%%%%%%%

\subsubsection*{"Slow" and "fast" solutions}

Let us first consider the Orr Sommerfeld operator alone, without any boundary condition.
It is a fourth order differential operator, therefore we have to construct four independent solutions to completely describe its solutions.
Let us first study their behavior at infinity.
At infinity $OS$ "degenerates" into
$$
(U_+  - c)  (\partial_y^2 - \alpha^2) \psi - \eps (\partial_y^2 - \alpha^2)^2 \psi = 0
$$
or equivalently 
$$
\Bigl[ (U_+ - c) - \eps (\partial_y^2 - \alpha^2) \Bigr] (\partial_y^2 - \alpha^2) \psi = 0,
$$
hence, as $y \to + \infty$, either $\psi \sim e^{\pm | \alpha | y}$ or
$\psi \sim e^{\pm \sigma y}$ where
$$
 \sigma= \sqrt{U_+ - c \over \eps} .
$$
In the sequel, $c$ will be small, hence $\sigma$ will be of order $| \eps^{-1/2} | \gg 1$.
It can thus be proven that there exists four independent solutions to this operator,
two with a "slow" behavior (like $e^{\pm | \alpha | y}$), called $\psi_{s,\pm}$, and two with a "fast" behavior
(like $e^{\pm \sigma y}$), called $\psi_{f,\pm}$.
The "-" subscript refers to solutions which go to $0$ at infinity and the "+" subscript to solutions diverging as $y$ goes
to $+ \infty$.

%%%

\subsubsection*{"Fast" solutions}

%%%

"Fast" solutions have very large derivatives. For these solutions, 
Orr Sommerfeld may be approximated by its higher order derivatives, namely by
\begin{equation} \label{modAiry}
(U_s - c) \partial_y^2 - \varepsilon \partial_z^4 = Airy  \, \partial_y^2,
\end{equation}
where $Airy$ is the modified Airy operator 
$$
Airy =  (U_s - c) - \varepsilon \partial_y^2.
$$
More precisely,
$$
OS = Airy  \,  \partial_y^2 + Rem,
$$
where it turns out that the remainder operator $Rem$ may be treated as a perturbative term.

The first step is thus to construct solutions of $Airy \, \, \partial_y^2 = 0$, or, up to two integrations in $y$, to construct
solutions of $Airy = 0$. Near $y_c$, the $Airy$ operator may further be approximated by the classical Airy operator
\begin{equation} \label{linAiry}
U_s'(y_c) (y - y_c) - \varepsilon \partial_y^2 .
\end{equation}
Solutions of (\ref{linAiry}) may be explicitly expressed as combinations of the classical Airy's functions
$Ai$ and $Bi$, which are solutions of the classical Airy equation 
$$
y \psi = \partial_y^2 \psi.
$$
More precisely, $Ai(\gamma (y - y_c))$ and $Bi(\gamma (y - y_c))$ are solutions of (\ref{linAiry}) provided
\begin{equation} \label{gamma}
\gamma = \Bigl( {i \alpha U_s'(y_c) \over \nu} \Bigr)^{1/3}.
\end{equation}
Now, solutions of $Airy = 0$ may be constructed, starting from solutions of (\ref{linAiry}), through the
so called Langer's transformation (which is a transformation of the phase and amplitude of the solution).
To go back to (\ref{modAiry}) we then have to integrate twice these solutions.
As a consequence, fast solutions to the Orr Sommerfeld equation may be expressed in terms of
second primitives of the classical Airy functions $Ai$ and $Bi$.
 Let us call $Ai(2,.)$ and $Bi(2,.)$ this second primitives of $Ai$ and $Bi$.
 The second step is to take into account the remainder term $Rem$, through an iterative scheme. We skip the technical details.
 
 The output of this construction is a complete description of $\psi_{\pm,f}$. In particular, at leading order,
 $$
 \psi_{-,f}(y) = Ai \Bigl( 2, \gamma (y - y_c) \Bigr) + O(\alpha).
 $$
Moreover, it can be proven that $\psi_{f,-}$ satisfies
 \begin{equation} \label{psif}
\psi_{f,-}(0) = Ai(2,-\gamma y_c) + O(\alpha)
\end{equation}
and
\begin{equation} \label{psifd}
\partial_y \psi_{f,-}(0) = \gamma Ai(1, - \gamma y_c) + O(1).
\end{equation}

%%%

\subsubsection*{"Slow" solutions}

%%%

For the "slow" solutions, as a first approximation, higher order derivatives may be neglected. We thus see
$OS$ as a perturbation of  Rayleigh's operator as $\nu$ goes to $0$, namely
$$
OS = Ray - \varepsilon Diff,
$$
where
$$
Diff(\psi) =   (\partial_z^2 - \alpha^2)^2 \psi .
$$
For "slow" solutions, $Diff$ may be treated as a small perturbation. This is indeed true for $\psi_{\alpha,c}^r$, which is smooth near $y_c$,
and we have
 $$
 OS_{\alpha,c,\nu}(\psi_{\alpha,c}^r) = O(\varepsilon).
 $$ 
 However, for $\psi_{\alpha,c}^s$ the situation is more delicate since $\eps \partial_y^4 \psi_{\alpha,c}^s$
behaves like $\eps (y - y_c)^{-3}$ near the critical layer. As a consequence, away from the critical layer, $\psi_{\alpha,c}^s$ is a "good" approximate
solution to $OS$. However, close to the critical layer, the viscous term $\eps Diff ( \psi_{\alpha,c}^s)$ can no longer be neglected.
A "boundary layer" at $y_c$, or more exactly an "inner layer", must be added at $y_c$ to "hide" the $\eps (y - y_c)^{-3}$ singularity.
We thus correct $\psi_{\alpha,c}^s$ by adding $\psi^c$, solution to
$$
Airy \, \partial_y^2 \psi^c = \varepsilon Diff(\psi_{\alpha,c}^s) .
$$
Note that $\psi^c$ may be explicitly computed using Airy functions. It turns out that $\psi_{\alpha,c}^s + \psi^c$ is then a "good enough" solution
to the Orr Sommerfeld equation.
It is then possible to construction two independent solutions to $OS$ (see \cite{GGN3}), and to combine them in order to have one which goes
to $0$ at infinity, called $\psi_{s,-}$.
Further computations \cite{Reid} lead to
\begin{equation} \label{psism}
\psi_{s,-}(0) = - c + \alpha {U_+^2 \over U_s'(0)} + O(\alpha^2)
\end{equation}
and
\begin{equation} \label{psisd}
\partial_y \psi_{s,-}(0) = U'_s(0) + O(\alpha).
\end{equation}
%\begin{equation} \label{psisdd}
%\partial_z^2 \psi_{s,-}(0) = O(1),
%\end{equation}
%where $U_+=\lim_{z\rightarrow +\infty}U_s(z)$.

%%

\subsection{Dispersion relation and related instability}

An eigenmode of Orr Sommerfeld equation (together with its boundary conditions)
 is a combination of these four particular solutions which goes to $0$ at infinity
and which vanishes, together with its first derivative, at $y  = 0$. As an eigenmode must go to $0$ as $y$ goes to 
infinity, it is a linear combination of $\psi_{f,-}$ and $\psi_{s,-}$ only. To match the boundary conditions at $y = 0$
there must exist nonzero constants $a$ and $b$ such that 
$$
a \psi_{f,-} (0) + b \psi_{s,-} (0) = 0
$$ and
$$
a \partial_y  \psi_{f,-} (0) + b \partial_y \psi_{s,-} (0) = 0.
$$
The dispersion relation is therefore
\begin{equation} \label{disper}
{   \psi_{f,-} (0) \over \partial_y  \psi_{f,-} (0)}
= {  \psi_{s,-} (0) \over \partial_y  \psi_{s,-} (0)}
\end{equation}
or, using the previous computations,
\begin{equation} \label{disper2}
\alpha {U_+^2 \over U_s'(0)^2}  - {c \over U_s'(0)} = \gamma^{-1} {Ai(2, - \gamma y_c) \over Ai(1,-\gamma  y_c)} 
+O(\alpha^2).
\end{equation}
To simplify the discussion we will focus on the particular case where $\alpha$ and $c$ are both of order $\nu^{1/4}$. It turns out 
that this is an area where instabilities occur, and we conjecture that it is in this region that the most
unstable instabilities may be found.
We rescale $\alpha$ and $c$ by $\nu^{1/4}$ and introduce
$$
\alpha = \alpha_0 \nu^{1/4}, \quad 
c = c_0 \nu^{1/4}, \quad
Z = \gamma y_c,
$$
which leads to
\begin{equation} \label{disper3}
\alpha_0 {U_+^2 \over U_s'(0)^2}  - {c_0 \over U_s'(0)} = 
 {1 \over (i \alpha_0 U'_s(y_c))^{1/3}}  {Ai(2, - Z) \over Ai(1,- Z)} + O(\nu^{1/4}).
\end{equation}
Note that as $U_s(y_c) = c$, 
$$
y_c = U_s'(0)^{-1} c + O(c)
$$
and
$$
Z =  \Bigl(i  U_s'(y_c) \Bigr)^{1/3} \alpha_0^{1/3} U_s'(0)^{-1} c_0 + O(\nu^{1/4}) .
$$
Note that the argument of $Z$ equals $\pi / 6$.
We then introduce the following function, called Tietjens function, of the real variable $z$
$$
Ti(z) = {Ai(2,z e^{- 5 i \pi / 6}) \over z e^{- 5 i \pi / 6} Ai(1,z e^{- 5 i \pi / 6})} .
$$
At first order the dispersion relation becomes
\begin{equation} \label{dispersionlimit}
\alpha_0 {U_+^2 \over U_s'(0)} = c_0 \Bigl[ 1 - Ti(- Z e^{5 i \pi / 6} ) \Bigr] .
\end{equation}
This dispersion relation can be numerically investigated. Note that it only depends on the limit $U_+$
of the horizontal velocity at infinity and on $U_s'(0)$. It can be proven that, provided $\alpha_0$ is large enough,
there exists $c_0$ with $\Im c_0 > 0$, leading to an unstable mode.
%(we refer to \cite{Bian2} for a detailed discussion of this point).

Let us numerically illustrate this dispersion relation in the case $U_s'(0) = U_+ = 1$.
Figure (a) shows the imaginary part of $c_0$ as a function of $\alpha_0$. We see that there exists 
a constant $\alpha_c$ such that $\Im c_0 > 0$ if $\alpha_0 > \alpha_c$ and $\Im c_0 < 0$ if
$\alpha_0 < \alpha_c$.
Figure (b) shows $\Re \lambda = \alpha_0 \Im c_0$ as a function of $\alpha_0$.
We see that there exists a unique global maximum to this function, at $\alpha_0 = \alpha_M \sim 2.8$.

\begin{figure}
\centerline{\includegraphics[width=4cm]{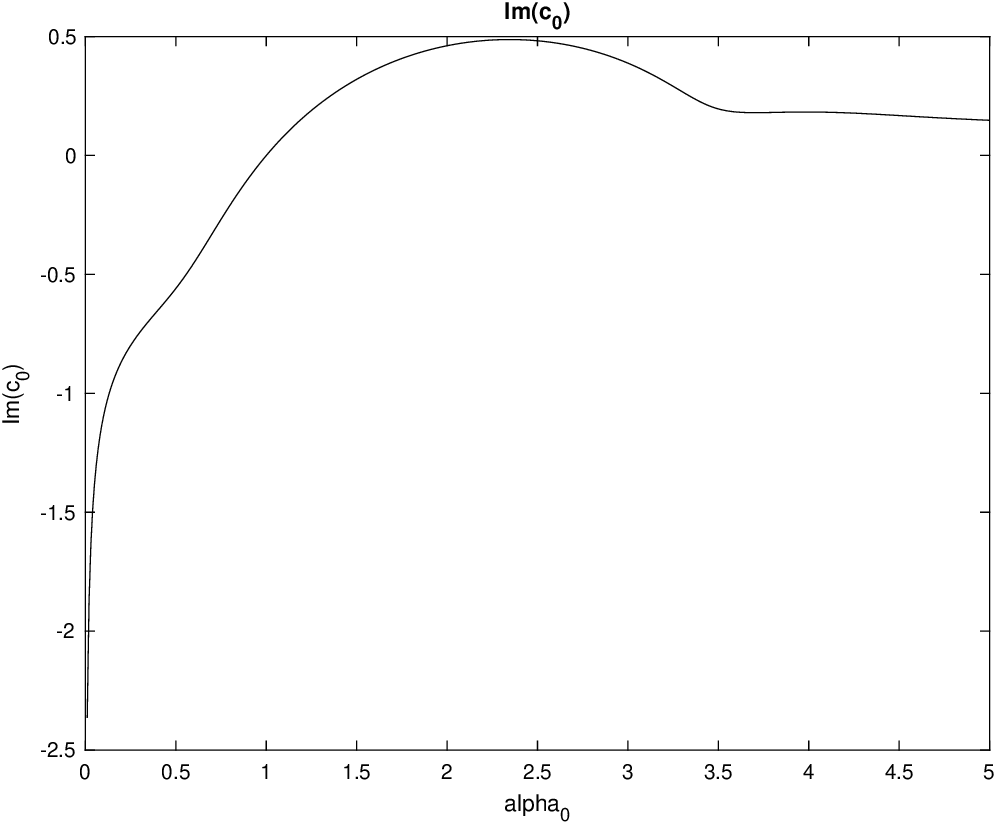}}
\caption{$\Im c_0$ as a function of $\alpha_0$ }
\label{valeurc}
\end{figure}

\begin{figure}
\centerline{\includegraphics[width=4cm]{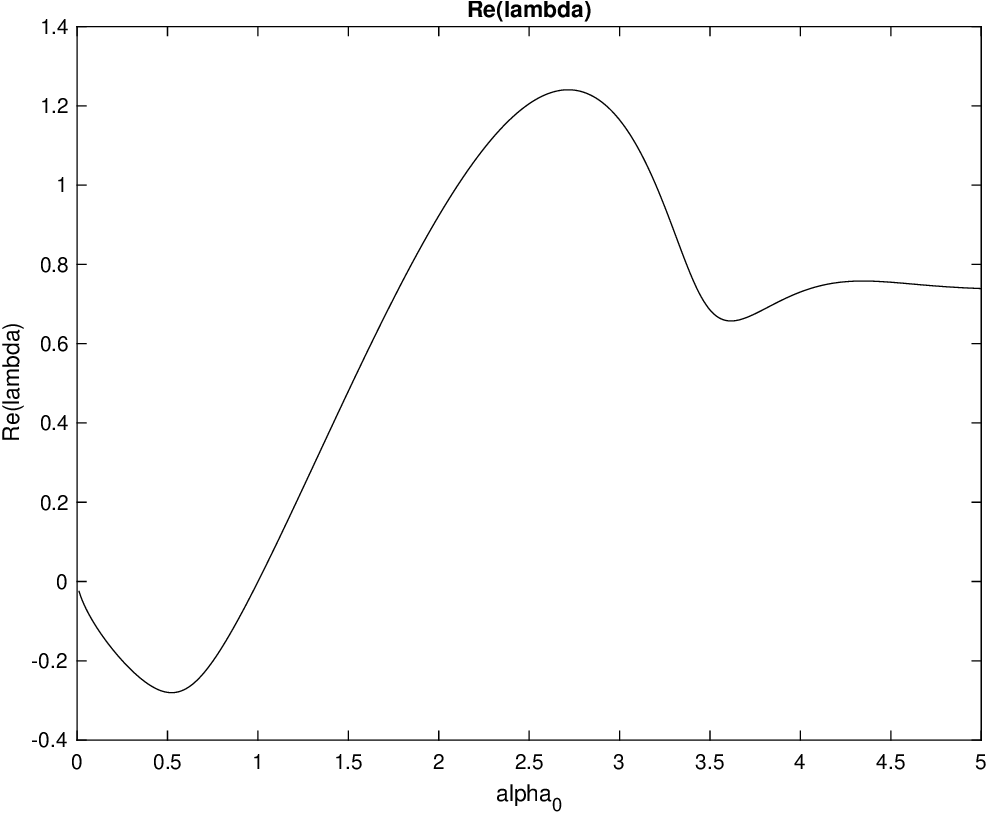}}
\caption{$\Re \lambda$ as a function of $\alpha_0$ }
\label{valeurlambda}
\end{figure}

Let us now detail the corresponding linear instability. Its stream function $\psi_{lin}$ is of the form
$$
\psi_{lin} = b \psi_{s,-} + a \psi_{f,-} .
$$
Choosing $b = 1$ we see that $a = O(\nu^{1/4})$, hence
\begin{equation} \label{psilin}
\psi_{lin}(y) = U_s(y) - c + \alpha {U_+^2 \over U_s'(0)} + a Ai \Bigl(2, \gamma (y - y_c) \Bigr) +  O(\nu^{1/2} ).
\end{equation}
The corresponding horizontal and vertical velocities $u_{lin}$ and $v_{lin}$ are given by
\begin{equation} \label{vlin}
u_{lin} = \partial_y \psi_{lin} = U_s'(y) +\gamma a Ai \Bigl(1, \gamma (y - y_c) \Bigr) + O(\nu^{1/4}),
\end{equation}
and 
\begin{equation} \label{ulin} 
v_{lin} = -i\alpha \psi_{lin} = O(\nu^{1/4}).
\end{equation}
Note that $\gamma a = O(1)$. The first term in (\ref{vlin}) may be seen as a "displacement velocity", corresponding
to a translation of $U_s$. The second term is of order $O(1)$ and located in the boundary layer, namely
within a distance $O(\nu^{1/4})$ to the boundary. Note that the vorticity
\begin{equation} \label{check-use-number}
\omega_{lin} = -(\partial_y^2 - \alpha^2) \psi_{lin}(y) 
= -U_s''(y) - \gamma^2 a Ai \Bigl( \gamma (y - y_c) \Bigr) 
\end{equation}
is large in the critical layer (of order $O(\nu^{-1/4})$).

%%%%%%%

\subsection{Green function for Orr Sommerfeld }

%%%%%%%

Let us now construction the Green function $G(x,y)$ for Orr Sommerfeld equation (see \cite{GN} for a detailed study). 
We search $G$ under the form
$$
G = G_i + G_b
$$
where $G_i$ takes care of the source term $\delta_x$ and where $G_b$ takes care of the boundary conditions.
We look for $G_i(x,y)$ of the form
$$
G_i(x,y) = a_+(x)  \phi_{s,+}(y) 
+ {b_+(x) }  {\phi_{f,+}(y) \over \phi_{f,+}(x)} 
\quad \hbox{for} \quad y < x,
$$
$$
G_i(x,y) = a_-(x)  \phi_{s,-}(y)
+ {b_-(x)} {\phi_{f,-}(y) \over \phi_{f,-}(x)} 
\quad \hbox{for} \quad y  > x,
$$
where $\phi_{f,\pm}(x)$ play the role of normalization constants. Let
$$
F_\pm = \phi_{f,\pm}(x) 
$$
and let 
$$
v(x) = (- a_-(x), a_+(x), - b_-(x), b_+(x) ) .
$$
By definition of a Green function, $G_i$, $\partial_y G_i$ and $\partial_y^2 G_i$ are continuous at $x = y$,
whereas $- \eps \partial_y^3 G_i$ has a unit jump at $x = y$. 
Let
\beq \label{matriceM}
M = \left( \begin{array}{cccc} 
\phi_{s,-} & \phi_{s,+} & \phi_{f,-} / F_-  & \phi_{f,+} /F_+  \cr
\partial_y \phi_{s,-}  & \partial_y\phi_{s,+}
& \partial_y\phi_{f,-} /  F_-  & \partial_y\phi_{f,+} / F_+  \cr
\partial_y^2 \phi_{s,-}  &\partial_y^2 \phi_{s,+}  
& \partial_y^2 \phi_{f,-} / F_-   & \partial_y^2 \phi_{f,+} /  F_+ \cr
 \partial_y^3 \phi_{s,-} &  \partial_y^3 \phi_{s,+}
& \partial_y^3 \phi_{f,-}  /  F_- &  \partial_y^3 \phi_{f,+} /  F_+ \cr 
\end{array} \right) ,
\eeq
where the various functions are evaluated at $x$. Then 
\beq \label{Mv}
M v = (0,0,0,-  \eps^{-1}) .
\eeq
Let $A$, $B$, $C$ and $D$ be the two by two matrices defined by
$$
M = \left( \begin{array}{cc} 
A & B \cr
C & D \cr \end{array} \right) .
$$
It can be proven (see \cite{GN}) that $D$ is invertible, $A$ is related to Rayleigh equation and that $C$ is small.
As a consequence,
$$
\widetilde M = \left( \begin{array}{cc} 
A^{-1} & - A^{-1} B D^{-1} \cr
0 & D^{-1} \cr \end{array} \right) 
$$
is an approximate inverse of $M$. By iteration, this allows to construct the true inverse of $M$ and to get bounds on it \cite{GN}.

We now add to $G_i$ another Green function $G_b$ to handle the boundary conditions.
We look for $G_b$ under the form
$$
G_b(y) = d_s \phi_{s,-}(y)  + d_f {\phi_{f,-}(y)  \over \phi_{f,-}(0)},
$$
where $\phi_{f,-}(0)$ in the denominator is a normalization constant, 
and look for $d_s$ and $d_f$ such that
\beq \label{Greenb1}
G_i(x,0) + G_b(0) = 
\partial_y G_i(x,0) + \partial_y G_b(0) =  0.
\eeq
Let
$$
N = \left( \begin{array}{cc} \phi_{s,-} & \phi_{f,-} / \phi_{f,-}(0) \cr
\partial_y \phi_{s,-} & \partial_y \phi_{f,-} / \phi_{f,-}(0) \cr \end{array} \right) ,
$$
the functions being evaluated at $y = 0$. Then (\ref{Greenb1}) can be rewritten as
$$
N d = - (G_i(x,0), \partial_y G_i(x,0)) 
$$
where $d = (d_s,d_f)$. 
We have
$$
N^{-1} = {1 \over \det(N)}  \left( \begin{array}{cc} \partial_y \phi_{f,-}(0) / \phi_{f,-}(0) & - 1 \cr
- \partial_y \phi_{s,-}(0)& \phi_{s,-}(0)  \end{array} \right) ,
$$
which ends the construction.

%%%%%%%

\subsection{Asymptotic behavior of linearised Navier Stokes}

%%%%%%%

As for Euler equation, using contour integral, we have
\beq
e^{L_{NS} t} \psi_0 = {\alpha \over 2  \pi} \int_\Gamma e^{i \alpha c t} Orr_{\alpha,c,\nu}^{-1} \psi_0 \, dc ,
\eeq
where
\beq
OS^{-1}_{\alpha,c,\nu}(y) = \int G(x,y) \psi_0(x) \, dx
\eeq
where $G$ is the Green function constructed in the previous paragraph.
Combining these two identities we get

\begin{theorem}
Let $U_s$ be a monotonic, concave, and analytic shear layer profile, such that $U_s(0) = 0$ and such that $U_s(y)$ converges exponentially
fast to some constant $U_+$ as $y \to + \infty$. Then for any small $\alpha$, 
\beq \label{asymptotic}
v_\alpha (t) = \exp( L_{NS,\alpha} t)  P_{unstable,\alpha} v(0) + O(e^{- C \nu^{1/2} \langle t \rangle}) ,
\eeq
where $P_{unstable,\alpha}$ is the projection on unstable modes.

Moreover, there exists $C_-(\nu)$ and $C_+(\nu)$, converging to non vanishing limits as $\nu \to 0$,
 such that $P_{unstable,\alpha} = 0$ if $| \alpha | \le C_-(\nu) \nu^{1/4}$
or $| \alpha | \ge C_+(\nu) \nu^{1/6}$. For $C_-(\nu) \nu^{1/4} \le | \alpha | \le C_+(\nu) \nu^{1/6}$, then
$$
P_{unstable,\alpha}(v_0) = e^{\lambda(\alpha) t} (v_0 | v(\alpha)) + c.c.,
$$
where $v(\alpha)$ is an eigenvector of its adjoint $L_{NS}^\star$ with the same eigenvalue, and 
$$
\sup_\alpha | \Re \lambda(\alpha) | \sim C_2 \nu^{1/2}
$$
as $\nu \to 0$.
\end{theorem}

The proof of this result is detailed in \cite{Bian3}. The main idea is that, thanks to the viscosity, there is no longer any singularity when
$c$ is on the range of $U_s$. We can thus shift $\Gamma$ to negative $\Im c$. However, the Airy functions are increasing when the
imaginary part of their argument is negative, and we can only shift $\Gamma$ by $\nu^{1/4}$ in $\{ \Im c < 0 \}$, which gives a small decay rate,
of magnitude $C \nu^{1/4}$.
Note that, in strong contrast with linearised Euler equation, there exist growing modes for particular values of $\alpha$: a shear monotonic
shear layer is always linearly unstable provided the viscosity is small enough.
This is somehow paradoxical: adding viscosity destabilises the flow.

%%%%%%%%%%%%%%%

\section{Euler unstable shear layers}

%%%%%%%%%%%%%%%

In this section, we begin the discussion of Theorem \ref{theomain}, and construct a sequence of solutions $V^\nu$ which exhibits the
scale $\nu^{3/4}$. The construction of the scale $\nu^{13/16}$ will be done in section $7$. The rigorous justification of this first step
can be found in \cite{GGN3}.

We note that $V_0^\nu$ has a vertical scale $\nu^{1/2}$. The first step is to rescale time and space to this scale, namely to introduce
$$
T_1 = {t \over \nu^{1/2}}, \qquad 
X_1 = {x \over \nu^{1/2}}, \qquad
Y_1 = {y \over \nu^{1/2}} .
$$
Navier Stokes remains unchanged by this change of variables, except that the viscosity is now
$$
\nu_1 = \nu^{1/2} .
$$
From now on we work in these rescaled variables.
The starting point is a profile $U_s(0,\cdot)$ which is spectrally unstable for linearised Euler equations, namely a profile $U_s(0,\cdot)$
such that there exists $\alpha$, $c_{1,0}$ and $\phi_{1,0}$, solutions of the Rayleigh equation 
\beq \label{Ray}
\Bigl( U_s(0,Y_1) - c_{1,0} \Bigr) (\partial_{Y_1}^2 - \alpha^2) \phi_{1,0} =  U_s''(0,Y_1) \phi_{1,0} ,
\eeq
with
\beq \label{Ray2}
\phi_{1,0}(0) = 0
\eeq
and with $\Im c_{1,0} > 0$ (unstable mode).

We start from $(c_{1,0},\phi_{1,0})$ of (\ref{Ray}) to construct a solution of Orr Sommerfeld equation
and thus to build an instability for Navier Stokes equations.
Note that $\phi_{1,0}$  does not fit the boundary condition 
$\partial_y \phi(0) = 0$.
We therefore have to add a "boundary layer" to $\phi_{1,0}$ in order to ensure this boundary layer condition.
Its size must be such that $\nu_1 \partial_{Y_1}^4$ balances the term in $\partial_{Y_1}^2$, which leads to a size of order $\nu_1^{1/2}$
and to look for solutions of  Orr Sommerfeld equations of the form
\beq \label{exp1}
\phi(Y_1) = \sum_{n \ge 0} \nu_1^{n/2} \phi_n^{int}(Y_1) + \sum_{n \ge 1} \nu_1^{n/2} \phi_n^{bl}(\nu_1^{-1/2} Y_1)  
\eeq
together with the expansion
\beq \label{exp1bis}
c_1 = \sum_n \nu_1^{n/2} c_{1,n}.
\eeq
In this expansion $\phi_n^{int}$ refers to an "interior" behaviour, and $\phi_n^{bl}$ to a "boundary layer" behaviour. 
Moreover we start with $\phi_0^{int} = \phi_{1,0}$. 
The first boundary layer profile satisfies
\beq \label{hie2}
\nu_1 \partial_{Y_1}^4 \phi_1^{bl} = (U_s(0,0) - c_1)  \partial_Y^2 \phi_1^{bl} ,
\eeq
together with the boundary condition
\beq \label{hie3}
\partial_{Y_1} \phi_1^{bl}(0) = - \partial_{Y_1} \phi_0^{int}(0).
\eeq
Hence
$$
\phi_1^{bl}( \nu_1^{-1/2}  Y_1) 
= - \partial_{Y_1} \phi_0^{int}(0) \nu_1^{1/2} \mu^{-1} \Bigl( 1   - e^{- \nu_1^{-1/2} \mu Y_1} \Bigr), \qquad \mu = (- c_1) ^{1/2}.
$$
Let us now compute the corresponding velocity field. Using
$$
(u^{bl}, v^{bl} ) = \nabla_{X_1,Y_1}^\perp \Bigl( e^{i \alpha_1 (X_1 - c_1 T_1)} \phi^{bl}(\nu_1^{-1/2} Y_1) \Bigr),
$$
we see that $u^{bl}$ and $v^{bl}$ have the following asymptotic expansions
$$ 
u^{bl} = -  e^{i \alpha_1 (X_1 - c_1 T_1)} e^{- \nu_1^{-1/2} \mu Y_1} \partial_{Y_1} \phi_0^{int}(0) + \hbox{c.c.} + O(\nu_1^{1/2})
$$
and
 $$
 v^{bl} = - i \alpha_1 \nu_1^{1/2} \mu^{-1} 
 e^{i \alpha_1 (X_1 - c_1 T_1)}
 \Bigl( 1   - e^{- \nu_1^{-1/2} \mu Y_1} \Bigr)   \partial_{Y_1} \phi_0^{int}(0) 
  + \hbox{c.c.} + O(\nu_1).
$$
We note that the horizontal speed in the boundary layer is of order $O(1)$ and exactly compensates the interior component
of the velocity at the boundary. The vertical speed is of size $O(\nu_1^{1/2})$ and vanishes as $\nu \to 0$.

The construction of \cite{GN1} exactly relies on this formal analysis and fully justifies it.
It provides the existence of a solution $V_1^\nu$ of Navier Stokes equation with forcing term $F^\nu$
which is of the form $V_1^\nu = (u^\nu_1,u^\nu_2)$ with
$$
u_1^\nu(t,x,y) =  U_s(T,Y_1) + \nu^N e^{i \alpha_1 (X_1 - c_1 T_1)} \partial_{Y_1} \phi_0^{int}(Y_1) 
- \nu^N e^{i \alpha_1 (X_1 - c_1 T_1)} e^{- \nu_1^{-1/2} \mu Y_1} \partial_{Y_1} \phi_0^{int}(0)
$$
$$
 + \hbox{c.c.} + \cdots,
$$
$$
v_1^\nu = - i \nu^N \alpha_1 \nu_1^{1/2} \mu^{-1}  e^{i \alpha_1 (X_1 - c_1 T_1)} 
\Bigl( 1   - e^{- \nu_1^{-1/2} \mu Y_1} \Bigr)   \partial_{Y_1} \phi_0^{int}(0) 
  + \hbox{c.c.} + \cdots.
$$
This expansion may be justified till the perturbation reaches an order $O(1)$, namely becomes of the order of the shear flow itself.

Note that in the original $y$ variable, the size of this new boundary layer is of order $\nu^{3/4}$.
Therefore at this stage, we have {\it three} vertical scales:
\begin{itemize}
\item[-]
$O(1)$ (characteristic size of the domain). At this scale the flow is constant and equals $(U_+,0)$.
\item[-]
$O(\nu^{1/2})$ (characteristic size of the shear layer profile or boundary layer profile). At this scale
the flow is the sum of $(U_s(Y_1),0)$ and of an exponentially growing periodic perturbation, solution of Rayleigh equation.
This  instability creates a velocity at the boundary which periodically changes sign in space and time, in contradiction with
Dirichlet boundary condition. 
\item[-]
$O(\nu^{3/4})$ (size of the sublayer of the unstable mode).
 The horizontal part of the velocity in this sublayer recovers Dirichlet boundary condition. The vertical
one ensures incompressibility by creating a small flow of order $O(\nu^{1/4})$, which leaves or enters this sublayer.  
\end{itemize}
There are {\it two} horizontal sizes, namely
\begin{itemize}
\item[-] $O(1)$ (size of the domain),
\item[-] $O(\nu^{1/2})$ (horizontal periodicity of the instability). Note that this size is absent in Prandtl's boundary layer.
Instability creates a small scale in $x$, which is absent from Prandtl's Ansatz. This scale is also "killed" if we assume
analyticity of the initial data.
\end{itemize}
Moreover we have {\it two} times scales
\begin{itemize}
\item[-] $O(1)$ (characteristic time scale of the evolution of $U_s$)
\item[-] $O(\nu^{1/2})$ (characteristic time scale of the instability in the boundary layer) 
\end{itemize}
All these characteristic scales appear in classical physics books like \cite{Reid}.
As we will see in section $7$, a third time scale will appear, linked with the instability of the sublayer of size
$O(\nu^{3/4})$, together with a new and unexpected "sub-sub-layer".

%%%%%%%%%%%%%%%

\section{Euler stable shear layers}

%%%%%%%%%%%%%%%

In this section we discuss Theorem \ref{theofirst2}. As in the previous section we introduce $T_1$, $X_1$ and $Y_1$.
However, as $U_s(0,\cdot)$ is stable for linear Euler equation,
we can not build an instability starting from Rayleigh. Instead we use the instability built in subsection $4.2$. For $\alpha_1$ in the range
$$
A_1 \nu_1^{1/4} \le \alpha_1 \le A_2 \nu_1^{1/6} 
$$
where $A_1$ and $A_2$ are two constants, there exists unstable modes of Orr Sommerfeld equations with a growing rate
$$
\alpha_1 \Im c_1 \sim \nu_1^{1/2} \sim \nu^{1/4}.
$$
This leads to instabilities within times $T_{insta,1}$ of order $\nu^{-1/4}$ in rescaled time, and therefore within real times of order
$$
T_{insta,1} \sim \nu^{-1/4} \nu^{1/2} \sim \nu^{1/4} \ll 1.
$$
Note that the unstable mode $\psi_1$ itself has a complex vertical structure, with a critical layer at $y_c$ with
$U_s^0(y_c) = c_1$, namely $y_c \sim \nu^{1/8}$ in rescaled variable, or $\nu^{1/8} \nu^{1/2} = \nu^{5/8}$ in original variables..
It also has a "recirculation" layer of size $\alpha_1^{-1} \sim \nu^{-1/4}$ in rescaled variable, or $\nu^{1/4}$ in original variables.

We note that there exist {\it four} vertical scales, namely
\begin{itemize}
\item[-] $O(1)$, characteristic size of the domain,
\item[-]  $O(\nu^{3/8})$, characteristic size of the recirculation layer,
\item[-]  $O(\nu^{1/2})$, characteristic size of $U_s$ (Prandtl's boundary layer size),
\item[-]  $O(\nu^{5/8})$, characteristic size of the critical layer (new layer created by the instability).
\end{itemize}
There exist {\it two} horizontal scales
\begin{itemize}
\item[-]  $O(1)$, characteristic size of the domain
\item[-]  $O(\nu^{1/8})$, periodicity of the instability arising in the layer.
\end{itemize}
There exits {\it two} time scales
\begin{itemize}
\item[-]  $O(1)$, characteristic time of the evolution of $V^\nu$,
\item[-]  $O(\nu^{1/4})$, characteristic time of the growth of the instability.
\end{itemize}
A crucial point is to investigate the size that these perturbations may reach. Perturbations as constructed in section $2$, namely
based on instabilities of Euler equations, reach a size $O(1)$. However the situation is more delicate for instabilities constructed
in section $3$, namely purely viscous ones, since their growth is very small, namely over times of order $\nu^{-1/8}$ in rescaled time.

Let $\phi(T)$ be the size of the instability. We observe that quadratic interactions have a different wavenumber, namely $\pm 2 \alpha_1$,
thus the feedback only occurs with cubic interactions, which have wave numbers $\pm \alpha_1$ and $\pm 3 \alpha_1$.
Thus, if we follow a bifurcation approach, $\phi(t)$ approximately satisfies an equation of the form
\beq \label{bifurcation}
\dot \phi(T) = \alpha_1 c_1 \phi(T) + A | \phi(T) |^2 \phi(T) + \cdots,
\eeq
with $\Im (\alpha_1 c_1)$ or order $\nu^{1/8}$.
The evolution of $\phi(T)$ then depends on the sign of $\Re A$. 

\begin{itemize}

\item[-]  If $\Re A < 0$, then the cubic interactions tend to "tame" the instability. In this case it is possible to prove that
$\phi(T)$ reaches a magnitude $\nu^{1/16}$, but not more.

\item[-]  If $\Re A = 0$, then cubic interactions cancel, and we have to go up to quintic interactions. Then $\phi(T)$ is expected
to reach $\nu^{1/32}$.

\item[-]  If $\Re A > 0$, then there is no obstacle for $\phi(T)$ to reach an order $O(1)$.

\end{itemize}
Note that $A$ is the projection on the unstable mode of the cubic interaction of this unstable mode with itself.
Preliminary numerical computations \cite{Bian2} indicate that we are in the first case. 
It is possible to prove that the perturbation reaches at least a size $O(\nu^{1/16})$.
It therefore seems that we are in a bifurcation context, of Hopf's type, which may lead to the onset of rolls, a direction of research
which deserves further investigations.

%%%%%%%%%%%%%%%

\section{Multiple layers}

%%%%%%%%%%%%%%%

We now turn to the discussion of the $\nu^{13/16}$ sublayer of Theorem \ref{theomain} and go on with section $5$.
In the section $5$ we have constructed a sublayer of size $\nu^{3/4}$ in which the typical velocity reaches a magnitude $O(1)$.
We can construct a Reynolds number for this $\nu^{3/4}$ sublayer by combining a typical length (its size, namely
$\nu^{3/4}$), a typical velocity (the magnitude of $V_1^\nu$, namely $U_+$) and the viscosity ($\nu$). This leads
to a Reynolds number for this sublayer of order
$$
Re_2 \sim {\nu^{3/4} U_+ \over \nu} \sim \nu^{-1/4} .
$$
Hence, the Reynolds number of the sublayer goes to infinity as $\nu$ goes to $0$.
But  any shear layer is linearly unstable at high Reynolds number.
Thus, we expect linear instabilities to appear in this sublayer.

To study this instability we consider $V^\nu$, as constructed in section $5$, at a time $T_1^\nu$ where the $\nu^{3/4}$ 
layer has fully developed, i.e. has reached a magnitude $O(1)$. We consider this time as the new initial time.
We  then rescale time and space in order to focus on the $\nu^{3/4}$ layer, and introduce
$$
T_2 = \nu^{-3/4}t, \quad X_2 = \nu^{-3/4} x, \quad Y_2 = \nu^{-3/4} y .
$$
After rescaling, Navier Stokes equations remain unchanged, excepted for the viscosity which is now
$$
\nu_2  =  \nu^{1/4} \sim {1 \over Re_2}.
$$
Let us now describe the flow $V^\nu(T_1^\nu,\cdot,\cdot)$ near $x = 0$, in the rescaled variables $X_2$ and $Y_2$, namely let us describe
$V^\nu(T_1^\nu,\nu^{3/4} X_2, \nu^{3/4} Y_2)$. We recall that $V^\nu(T_1^\nu,\cdot,\cdot)$ has scales $1$ and $\nu^{1/2}$ in $x$
and $\nu^{1/2}$ and $\nu^{3/4}$ in $y$. Hence, in the variables $X_2$ and $Y_2$, $V^\nu(T_1^\nu,\cdot,\cdot)$ changes
on scales of order $\nu^{-1/4}$ in the $X_2$ variable, namely is almost constant. In the $Y_2$ variable it has the scales $1$ and
$\nu^{-1/4}$. In other words
\beq \label{description}
V^\nu(T_1^\nu + \nu^{3/4} T_2,\nu^{3/4} X_2, \nu^{3/4} Y_2) = U_1 + O(\nu^{1/4} X_2 + \nu^{1/4} Y_2)
\eeq
where
$$
U_1 = (U_1^\nu,0),
$$
with
$$
U_1^\nu(X_2,Y_2) = -  e^{i \alpha_1 \nu^{1/4} (X_2 - c_2 T_2)} \Bigl[ 1 - e^{-\mu Y_2} \Bigr]  \partial_{Z_1} \phi_0^{int}(0) + \hbox{c.c.}.
$$
Note that the boundary layer has an exponential behaviour in $Y_2$, 
and periodically depends  on the $X_2$ variable, with a very large spatial period of order $O(\nu^{-1/4})$ in this variable.
We note that $U_1^\nu(0,\cdot)$ has no inflection point and is therefore  stable for Rayleigh's equation.
However, $U_1^\nu$ depends on $X_2$, which was not the case in section $5$ and $6$. 

The first step is to "freeze" the $X_2$ variable at $X_2 = 0$ and to look for an unstable mode of the form
$$
v = \nabla^\perp \Bigl(  e^{\alpha_2 ( X_2 - c_2 T_2)} \psi_2(Y_2) \Bigr) 
$$
for the profile $U_1^\nu(0,Y_2)$. Then $\psi_2$ must satisfy the following Orr Sommerfeld equations
\beq \label{alpha2insta}
\eps_2 (\partial_{Y_2}^2 - \alpha_2^2)^2 \psi_2 = (U_1^\nu(0,Y_2) - c_2) (\partial_{Y_2}^2 - \alpha_2^2) \psi_2 
- \partial_{Y_1}^2 U_1^\nu(0,Y_2) \psi_2 .
\eeq
As proved in \cite{GGN3} and recalled in section $6$, for $\alpha_2$ in the range
$$
A_1 \nu_2^{1/4} \le \alpha_2 \le A_2 \nu_2^{1/6} 
$$
where $A_1$ and $A_2$ are two constants, there exist unstable modes to (\ref{alpha2insta}) with a growth rate
$$
\alpha_2 \Im c_2 \sim \nu_2^{1/2} \sim \nu^{1/8}.
$$
This leads to instabilities within times $T_{insta,2}$ of order $\nu^{-1/8}$ in rescaled time, and therefore within real times of order
$$
T_{insta,2} \sim \nu^{-1/8} \nu^{3/4} \sim \nu^{5/8} \ll \nu^{1/2}.
$$
Hence the instabilities of this sublayer grow faster than the initial Prandtl instability.
Note that $T_{insta,2}$ is large in $T_2$ scale. In the rescaled variable $T_2$, the growth of the instability is slow, mainly because
this instability has large structures in $X_2$.
Let us choose $\alpha_2$ of order $\nu_2^{1/4} \sim \nu^{1/16}$ to fix the ideas. This corresponds to structures in the horizontal variable of size
$\nu^{-1/16}$ in $X_2$ variable, or $\nu^{3/4 - 1/16} = \nu^{11/16}$ in the initial $x$ variable.
Note that $c_2$ is of order $\nu_2^{1/4} \sim \nu^{1/16}$.

Note that this unstable mode $\psi_2$ itself has a complex vertical structure, composed of three scales. Of course $Y_2 \sim 1$,
but also a "critical layer" and a "recirculation layer". 
First the "critical layer" is defined by 
$$
U_1^\nu(Z_c) = c_2.
$$
We note that
$$ 
Z_c  \sim \nu^{1/16} ,
$$
or, in initial variables, 
$$
z_c \sim \nu^{3/4} \nu^{1/16} \sim \nu^{13/16}.
$$
There is also a "recirculation layer", of size $\alpha_2^{-1}$ in $Y_2$ variable, corresponding to the kernel of $\partial_{Y_2}^2 - \alpha_2^2$.
It size is $\nu_2^{-1/4} \nu^{3/4} = \nu^{-1/16} \nu^{3/4} = \nu^{11/16} \ll \nu^{1/2}$. It is thus smaller that Prandtl's layer.
Note that this vertical scale corresponds to the scale in $X_2$: the instability creates large scale rolls to handle the flow in the
critical layer.

In the previous paragraph we have constructed an instability for  $(U_1^\nu,0)$ and not for the real layer $(U_1^\nu,V_1^\nu)$.
We have now to check whether the instability can survive to the spatial dependency of the background, and to the small upward velocity. 
The difference between $U_1$ and $(U_1^\nu,V_1^\nu)$
has two aspects: we have a dependence on $\nu^{1/4} Y_2$ and an additional vertical velocity $V_1^\nu$ is of order $\nu^{1/4}$.

Let us discuss the first aspect. We will pick up the most unstable mode $\alpha_2$, namely the mode with the largest
growth rate $\alpha_2 \Im c_2$, and construct a "wave packet" of instabilities near this $\alpha_2$. 
Our analysis of growth is valid as long as this wave packet does not travel to areas where the underlying flow notably changes.
Hence it is only valid as long as it travels  on length scales much smaller than $\nu^{-1/4}$ in rescaled variables, hence on length scales
much smaller than $\nu^{1/2}$ in original variables.

The wave packet travels with group velocity
$$
\sigma_2 = {\partial(\alpha_2 c_2) \over \partial \alpha_2} 
=  c_2 + \alpha_2  \frac{\partial c_2}{\partial \alpha_2}  \sim \nu_2^{1/4} \sim \nu^{1/16}.
$$
The maximum travel distance, within the instability time $T_{insta,2}$, is 
$$
Y_2  = \sigma_2 T_{insta,2}\approx \nu^{1/16} \nu^{-1/8} = \nu^{-1/16} ,
$$
or in the original coordinates, the maximum travel distance is
$$ 
\nu^{3/4} \nu^{-1/16} = \nu^{11/16}  \ll \nu^{1/2}.
$$
As a consequence we may "localise" the instability and construct an instability for the real layer $(U_1^\nu,V_1^\nu)$ starting
from an instability of $(U_1^\nu,0)$ (see \cite{Bian4} for a rigorous proof).
This ends the discussion of Theorem \ref{theomain}.
It turns out that this instability is too small in magnitude to allow the construction of an other sublayer, and to iterate the construction.

%%%%%%%%%%%%%%%

\section{Discussion}

%%%%%%%%%%%%%%%

As mentioned in the introduction, the various results on the behavior of solutions to Navier Stokes equations at low viscosity are somehow paradoxical.

\medskip

A first series of results states that Prandtl's analysis is true in small time and for analytic data. This is exactly the physical statement. However
the smoothness of the data is too strong, and hides the underlying existing instability.
Let us give a simple example. Let us consider the complex transport equation
$$
\partial_t \phi + i \partial_x \phi = 0.
$$
Then the Fourier transform satisfies
\beq \label{ill}
\partial_t \phi(t,\xi) = \xi \phi(t,\xi),
\eeq
and of course
$$
\phi(t,\xi) = e^{ \xi  t} \phi(0,\xi).
$$
In particular (\ref{ill}) is ill posed in any Sobolev space. However if we introduce the analytic norm
$$
\| \phi (t,\cdot) \|_\sigma = \sup | \phi(t,\xi) | e^{- \sigma | \xi |} 
$$
then we see that
$$
\| \phi(t,\cdot) \|_{\sigma - t} \le \| \phi(0,\cdot) \|_{\sigma},
$$
hence $\phi$ is locally well posed in analytic space. The analyticity simply "kills" the instability, which is hidden until $t = \sigma$ where
the radius of analyticity vanishes.

\medskip

Let us now discuss the multilayers instability. A first remark is that a shear layer is unstable provided its Reynolds number is large enough,
usually of order $1000$ to  $10000$ for an exponential profile. We note that the $\nu^{3/4}$ appears if $Re_1 > 1000$ to $10000$, which means that the initial
Reynolds number is larger than $10^6$ or even $10^8$. The third layer appears if $Re_2 > 1000$ to $10000$, namely if the Reynolds number is of order
$10^9$ or even $10^{12}$, namely at very high Reynolds numbers. In practise it is impossible to observe laminar flows at such Reynolds number, since
small perturbations on the walls would immediately lead to a turbulent behaviour.
Thus mathematical multiple layers instabilities appear for very large Reynolds numbers, beyond experiments.

%\medskip
%
%Thus, in experiments, such a third sublayer can not be observed and occurs at Reynolds numbers which are too high.
%In other words, in physics, the viscosity may be small, but it is never very small, and never goes to $0$.
%The mathematical limit $\nu \to 0$ is thus irrelevant to describe fluids at high Reynolds number.
%
%\medskip
%
%We have to replace the notion of "qualitative" behaviour by the notion of "quantitative" behaviour, namely to describe the
%solutions for small, but fixed parameters.
%
%\medskip
%
%A possibility is to quantify the stability of the solution.
%
%\medskip
%
%{\it Claim}
%
%\medskip
%
%
%For any $u^\nu$ and $v^\nu$ solutions of Navier Stokes, 
%provided the Reynolds number in the boundary layer is smaller than a critical Reynolds $Re_c$ number everywhere, we have
%$$
%\| u^\nu(t) - v^\nu(t) \|_{L^2}^2 \le \exp \Bigl( 2 \int_0^t \| \nabla u^\nu(\tau) \|_{L^\infty} \, d\tau \Bigr) \| u^\nu(0) - v^\nu(0) \|_{L^2}^2
%+ 2 \nu \Gamma \Bigl( \| u^\nu \|_{L^\infty}^2 + \| v^\nu \|_{L^\infty}^2 \Bigr) .
%$$
%where $\Gamma$ is the length of the boundary
%
%
%

%%%%%%%%%

\end{document}